\documentclass[11pt]{amsart}
\usepackage{amscd}      
\usepackage{amssymb}
\usepackage{amsmath}
\usepackage{xypic}      
\LaTeXdiagrams          
\usepackage[all,v2]{xy}
\usepackage{dsfont}
\xyoption{2cell} \UseAllTwocells \xyoption{frame} \CompileMatrices
\allowdisplaybreaks[3]
\usepackage{hyperref}
\usepackage{xcolor}

\usepackage{latexsym}
\usepackage{epsfig}
\usepackage{amsfonts}
\usepackage{enumerate}

\newtheorem{prop}{Proposition}[section]
\newtheorem{lem}[prop]{Lemma}

\newtheorem{thm}[prop]{Theorem}

\numberwithin{equation}{section}
\theoremstyle{definition}

\newtheorem{defn}[prop]{Definition}

\newtheorem{example}[prop]{Example}

\newtheorem{rmk}[prop]{Remark}

            {\nolinebreak $\Box$ \end{trivlist}}

\newcommand{\noprint}[1]{}

\renewcommand{\tilde}{\widetilde}

\newcommand{\sm}{\mbox{\tiny sm}}

\newcommand{\virt}{\mbox{\tiny virt}}
\newcommand{\red}{\mbox{\tiny red}}

\newcommand{\pt}{\mathop{pt}}

\renewcommand{\AA}{{\mathfrak A}}

\newcommand{\zz}{{\mathbb Z}}

\newcommand{\kk}{{\mathbb K}}
\newcommand{\vv}{{\mathbb V}}

\newcommand{\nn}{{\mathbb N}}

\renewcommand{\ll}{{\mathbb L}}
\newcommand{\qq}{{\mathbb Q}}
\newcommand{\pp}{{\mathbb P}}
\newcommand{\cc}{{\mathbb C}}
\newcommand{\rr}{{\mathbb R}}

\newcommand{\tT}{{\mathcal T}}

\newcommand{\sL}{{\mathcal L}}

\newcommand{\sH}{{\mathcal H}}
\newcommand{\sF}{{\mathcal F}}

\newcommand{\tX}{{\widetilde X}}
\newcommand{\tD}{{\widetilde D}}
\newcommand{\tY}{{\widetilde Y}}
\newcommand{\oO}{{\mathcal O}}

\newcommand{\A}{{\mathcal A}}
\newcommand{\M}{\mathcal{M}}

\newcommand{\CR}{\mbox{\tiny CR}}

\newcommand{\age}{\mbox{age}}
\newcommand{\cM}{\overline{\mathcal{M}}}
\newcommand{\NE}{\mbox{NE}}
\newcommand{\inv}{\mbox{inv}}
\newcommand{\bbT}{\mathbb{T}}

\newcommand{\bN}{{\mathbf N}}
\newcommand{\Cstar}{\mathbb{C}^{\times}}
\newcommand{\FM}{\mathbb{FM}}
\newcommand{\<}{{\langle}}
\newcommand{\sfy}{\mathsf{y}}

\newcommand{\sfp}{\mathsf{p}} 
\newcommand{\uu}{\mathbb{U}}

\newcommand{\rk}{\mathop{\rm rk}}

\newcommand{\ev}{\mathop{\rm ev}\nolimits}

\newcommand{\Hom}{\mathop{\rm Hom}\nolimits}

\newcommand{\Aut}{\mathop{\rm Aut}\nolimits}

\numberwithin{equation}{section}

\def\Label#1{\label{#1}{\tt [#1]}\phantom{h}}
\def\Label{\label}

\title[CRC implies the monodromy conjecture]{The Crepant Transformation conjecture implies the monodromy conjecture}

\author{Yunfeng Jiang}
\address{Department of Mathematics\\ University of Kansas\\ 405 Snow Hall 1460 Jayhawk Blvd\\Lawrence\\ KS 66045\\ USA}
\email{y.jiang@ku.edu}
\author{Hsian-Hua Tseng}
\address{Department of Mathematics\\ Ohio State University\\ 100 Math Tower,  231 West 18th Ave. \\Columbus, \\OH 43210\\ USA}
\email{hhtseng@math.ohio-state.edu}
\date{\today}

\begin{document}
\begin{abstract}
In this note we prove that the crepant transformation conjecture for a crepant birational transformation of Lawrence toric DM stacks studied in \cite{CIJ} implies  the monodromy conjecture for the associated wall crossing of the symplectic resolutions of hypertoric  stacks, due to Braverman, Maulik and Okounkov. 
\end{abstract}

\maketitle

\section{Introduction}

Let $X_1, X_2$ be two birationally equivalent smooth symplectic Deligne-Mumford (DM) stacks, which are symplectic resolutions of a singular symplectic stack.  Suppose by mirror symmetry they correspond to 
the two large radius points $0, \infty$ in the compactified K\"ahler moduli space $\mathcal{M}$. 
The derived categories of $X_1$ and $X_2$ are expected to be equivalent:
$$D^b(X_1)\cong D^b(X_2).$$
This equivalence can be given by a choice of path from $0$ to $\infty$ and thus gives a map 
$$\rho: \pi_{1}(\mathcal{M})\to \mbox{Aut}(D^b(X_i)).$$  
Also  this map $\rho$ is considered as a map in the level of K-theory
$$\rho: \pi_{1}(\mathcal{M})\to \mbox{Aut}(K^0(X_i)),$$ 
where $K^0(X_i)$ are the Grothendieck $K$-group of $X_i$. 
The monodromy conjecture for symplectic resolutions was formulated by Braverman-Maulik-Okounkov in \cite{BMO}, and it  
can be stated as follows: the monodromy of the quantum connection 
$\bigtriangledown$  for $X_i$ is the same as the above monodromy given by the equivalence in the K-theory.  
In the case of Hilbert scheme of points on the plane, in \cite{BO} Bezrukavnikov and Okounkov have proved that the monodromy of the quantum differential equation is isomorphic to $\rho$ in the level of K-theory. 

There are several candidates for the birationally equivalent smooth symplectic DM stacks.  For example, stratified Mukai flops
in \cite{Fu}, \cite{Cautis}, \cite{Kawamata}; and Mukai type flops of Nakajima quiver varieties \cite{MO}. 
In this paper we prove the conjecture for the crepant birational transformation of hypertoric DM stacks given by the  wall crossing induced by a single wall crossing of the associated Lawrence toric DM stacks. 

From \cite{CIJ}, a single wall crossing of Lawrence toric DM stacks is given by varying stability conditions in the GIT construction of the Lawrence toric DM stacks. There is a one-to-one correspondence between the  GIT data of Lawrence toric DM stacks and the extended Lawrence stacky fans in \cite{Jiang}.  
The wall crossing of hypertoric DM stacks is also given by varying the stability conditions in the GIT construction.  Generalizing the construction in \cite{CIJ},  we introduce the {\em extended stacky hyperplane arrangements} and define the hypertoric DM stacks associated with them. It turns out that the hypertoric DM stack associated with the extended stacky hyperplane arrangement is isomorphic to the hypertoric DM stack associated with the underlying  stacky hyperplane arrangement, see \S \ref{GIT:hypertoricDMstack}, and there is a one-to-one correspondence between the GIT data and the extended stacky hyperplane arrangements. 
Hence  the wall crossing of hypertoric DM stacks  gives  the Mukai type flops.

Let $X_+\dasharrow X_-$ be a crepant birational map between two smooth Lawrence toric DM stacks given by a single wall crossing in \cite{CIJ}.   They are derived equivalent, and the equivalence is given by the Fourier-Mukai transformation, see \cite{Kawamata} and \cite{CIJS}.   In \cite{CIJ}, the authors prove that the equivariant Fourier-Mukai transformation on the K-theory matches the analytic continuation of the  $I$-function, hence matches the analytic continuation of the quantum connections which are determined by the $I$-function.  Thus the genus zero crepant transformation conjecture due to Y. Ruan \cite{R} is proved.  
The wall crossing $X_+\dasharrow X_-$ implies  that the associated birational transformation $Y_+\dasharrow Y_-$ for the hypertoric DM stacks  is crepant. The derived categories of $Y_+$ and $Y_-$ are equivalent, and the Fourier-Mukai functor gives such an equivalence. 
We  prove that the Fourier-Mukai transformation matches the analytic continuation of quantum connections of the hypertoric 
DM stacks, which is induced by the analytic continuation of the associated Lawrence toric DM stacks, see Theorem \ref{main:theorem}.

Crepant birational transformation of hypertoric DM stacks is the local model of stratified Mukai type flops for general symplectic DM stacks. The authors in \cite{Cautis}, \cite{Cautis2}, \cite{Kawamata} have proved that their derived categories are equivalent, and the kernel is also given by the Fourier-Mukai  transformation.  The construction and calculation in this paper will play a role in the study for general Mukai type flops by degeneration method to the local models.

The paper is outlined as follows.  In \S \ref{wall:crossing} we study the GIT data of hypertoric DM stacks and construct the wall crossing. In \S \ref{Fourier-Mukai} we calculate the Fourier-Mukai transformation for the wall crossing of the hypertoric DM stacks; and in 
 \S \ref{analytic:continuation} we talk about the analytic continuation and prove the monodromy conjecture.  

\subsection*{Conventions}
In this paper we work entirely algebraically over the field of
complex numbers.  (Quantum ) cohomology and $K$-theory groups are taken with
complex coefficients. 
We refer to \cite{BCS} for the construction of Gale dual
$\beta^{\vee}: \mathbb{Z}^{m}\to DG(\beta)$ from 
$\beta: \mathbb{Z}^{m}\to N$. We denote by $N\to \overline{N}$ the natural map
modulo torsion.  
For a positive integer $m$, we use $[m]$ to represent the set 
$\{1,\cdots,m\}$.

\subsection*{Acknowledgments}

The authors would like to thank D. Edidin, N. Proudfoot and M. MacBreen for the correspondence of symplectic resolution and K-theory of hypertoric DM stacks.  Y. J. especially thanks Gufang Zhao to draw his attention to the MIT-Northeastern seminar series on quantum cohomology, geometric representation theory and the monodromy conjecture.  Both authors are partially  supported by Simons Foundation. 

\section{Wall crossing of  Hypertoric Deligne-Mumford Stacks}\label{wall:crossing}

In this section we prove that the single wall crossing of Lawrence toric DM stacks in sense of \cite[\S 5]{CIJ}
 gives rise to a wall crossing of hypertoric DM stacks. 
 
 \subsection{Lawrence toric DM stacks and the GIT construction}\label{sec:LawrencetoricDMstack}
 \begin{defn}\label{extended:Lawrence:stacky:fan}
 An \emph{$S$-extended Lawrence stacky fan} 
is a quadruple $\mathbf{\Sigma}_{L}= (\bN_{L},\Sigma_{L},\beta_L,S)$, where:
\begin{itemize} 
\item $\bN_{L}$ is a finitely generated abelian group (torsions allowed); 
\item $\Sigma_{L}$ is a rational Lawrence simplicial fan in $\bN\otimes \rr$ in sense of \cite[\S 4]{HS};   
\item $\beta \colon \zz^N \to \bN$ is a homomorphism; 
we write $b_i = \beta(e_i)\in \bN$ for the image of the $i$th standard 
basis vector $e_i\in\zz^N$,
and write $\overline{b}_i$ for the image of $b_i$ in $\bN\otimes \rr$; 
\item $S \subset \{1,\dots,N\}$ is a subset, such that $N=2n+|S|$ for a nonnegative integer $n$. 
\end{itemize} 
such that:
\begin{itemize} 
\item each one-dimensional cone of $\Sigma_{L}$
is spanned by $\overline{b}_i$ for a unique 
$i\in \{1,\dots,N\}\setminus S$, and
each $\overline{b}_i$ with 
$i\in \{1,\dots,N\} \setminus S$ spans a one-dimensional 
cone of $\Sigma_{L}$; 

\item for $i\in S$, $\overline{b}_i$ lies in the support $|\Sigma_{L}|$ 
of the fan. 
\end{itemize} 
\end{defn}
The vectors $b_i$ for $i\in S$ are called the \emph{extended vectors}. 
The \emph{Lawrence toric DM stack} associated to an extended Lawrence stacky fan $\mathbf{\Sigma}_{L}$ depends only on the underlying Lawrence stacky fan and is defined as
the quotient stack
\[ 
 X_{\mathbf{\Sigma}_{L}} := [ U/K ], \quad \text{with} \quad 
 U=\cc^{2n}\setminus \vv(I_{\Sigma_{L}})\times (\Cstar)^{|S|},
\]
where 
$I_{\Sigma_{L}}$ is the irrelevant ideal of the Lawrence fan $\Sigma_{L}$ and $K$ is a finitely generated abelian group, which  acts on $\cc^N$ through the data of extended Lawrence stacky fan.

 We require that the extended Lawrence stacky fan $\mathbf{\Sigma}_{L}$ 
satisfies the following condition: 
\begin{itemize} 
\item[(C1)] the map $\beta \colon \zz^N \to \bN$ is surjective. 
\end{itemize} 
This condition can be always achieved by adding  
enough extended vectors.

The Lawrence toric DM stack $X_{\mathbf{\Sigma}_{L}}$ 
is semi-projective and has a torus fixed point, see \cite{JT2}. 
We explain the GIT construction of $X_{\mathbf{\Sigma}_{L}}$ from the extended Lawrence
stacky fan $\mathbf{\Sigma}_{L}=(\bN, \Sigma_{L},\beta_{L},S)$ satisfying (C1).
We define a free $\zz$-module $\ll$ by the exact sequence
\begin{equation}\label{eq:exact}
  \xymatrix{
    0 \ar[r] &
    \ll \ar[r] &
    \zz^N \ar[r]^\beta & 
   \bN  \ar[r] & 
    0 }
\end{equation}
and define $K := \ll\otimes \Cstar$. 
The dual of \eqref{eq:exact} is an exact sequence:
\begin{equation} 
\label{eq:divseq}
\xymatrix{
  0 \ar[r] &
  \bN^{\vee} \ar[r] &
  (\zz^N)^{\vee} \ar[r]^{\beta^\vee}& 
  \ll^{\vee}}
\end{equation} 
and we define the character $D_i \in \ll^\vee$ of $K$ to be the image 
of the $i$th standard basis vector in $(\zz^N)^\vee$ under 
$\beta^\vee$.  We have 
$$(\zz^N)^\vee\cong (\zz^n\oplus \zz^n)^\vee\oplus\zz^{|S|}$$ and 
$$\{D_1,\cdots,D_N\}=\{D_1,\cdots,D_n, -D_1,\cdots, -D_n, D_{1},\cdots, D_{|S|}\}. $$
For $I\subset \{1,2,\cdots,N\}$, let 
$\sigma_I$ denote the cone in $\bN\otimes\rr$ generated by $\overline{b}_i$ for $i\in I$. 
Let $\overline{I}$ be the complement of $I\subset \{1,2,\cdot, N\}$.
Set
\[
\A:=\left\{I\subset \{1,2,\cdots,N\} \mid
S \subset I, \ \text{$\sigma_{\overline{I}}$ 
is a cone of $\Sigma_{L}$}\right\} .
\]
to be the collection of \emph{anticones}.
The \emph{stability condition} $\theta\in \ll^{\vee}\otimes\rr$ is taken to 
lie in $\bigcap_{I \in \A} \angle_{I}$, where
$$\angle_I = \big\{ \textstyle\sum_{i \in I} a_i D_i \mid \text{$a_i \in
      \rr$, $a_i > 0$} \big\} \subset
    \ll^\vee \otimes \rr.$$
The property of Lawrence toric fan $\Sigma_{L}$ ensures that this intersection is non-empty.  
We understand $\angle_\emptyset=\{0\}$. 
Let $$\A_\theta= 
    \big\{ I \subset \{1,2,\ldots,N\} : \theta \in \angle_I \big\}.$$
Then we check that $\A_\theta=\A$.     
Let 
\[
 U_{\theta}=\bigcup_{I\in \A_\theta}(\cc^\times)^{I}\times \cc^{\overline{I}}
 := (\cc^\times)^{I}\times \cc^{\overline{I}}=\{(z_1,\cdots,z_N) \in \cc^N \mid
 z_i\neq 0~ \mbox{for}~ i\in I\}.
\]
The stability condition $\theta$ satisfies the following assumptions:
  \begin{enumerate}
  \item[(A1)] $\{1,2,\ldots,N\} \in \A_\theta$;
  \item[(A2)] for each $I \in \A_\theta$, 
the set $\{D_i : i \in I\}$ spans $\ll^\vee \otimes \rr$ over $\rr$.
\end{enumerate}
(A1) ensures that $X_\theta$ is non-empty; 
(A2) ensures that $X_\theta$ is a DM stack. 
Under these assumptions, $\A_\theta$ is closed under 
enlargement of sets; i.e., if $I\in \A_\theta$ and 
$I\subset J$ then $J \in \A_\theta$. 
The {Lawrence toric DM stack} 
is the quotient stack
$X_{\mathbf{\Sigma}_{L}}=X_\theta=[U_\theta/K]$.
The \emph{GIT data} for $X_{\mathbf{\Sigma}_{L}}$ consists of
\begin{itemize} 
\item $K \cong (\Cstar)^r$, a connected torus of rank $r$; 
\item $\ll = \Hom(\Cstar,K)$, the cocharacter lattice of $K$; 
\item $D_1,\ldots,D_n, -D_1, \cdots, -D_n, D_{1},\cdots, D_{|S|} \in \ll^\vee = 
\Hom(K,\Cstar)$, the characters of $K$; 
\item stability condition $\theta\in  \ll^{\vee}\otimes\rr$.
\end{itemize}

Conversely, to obtain an extended Lawrence stacky fan from GIT data, 
consider the exact sequence
(\ref{eq:exact}). 
Let $b_i = \beta(e_i)\in \bN$ and $\overline{b}_i\in \bN\otimes \rr$ 
be as above. 
The extended Lawrence stacky fan
$\mathbf{\Sigma}_\theta=(\bN, \Sigma_\theta, \beta_{L},S)$
corresponding to our data 
consists of the group $\bN$ and the map $\beta_{L}$ defined 
above, together with a fan $\Sigma_\theta$ in $\bN\otimes\rr$ 
and $S$ given by 
\begin{align*} 
\Sigma_\theta  = \{\sigma_{I} : \overline{I} \in \A_\theta\}, 
 \qquad S  = \{ i \in \{1,\dots,N\} : \overline{\{i\}}  
\notin \A_\theta\}. 
\end{align*} 
The quotient construction in \cite[\S2]{Jiang} 
coincides with the GIT quotient construction, 
and therefore $X_\theta$ is the Lawrence toric DM
stack corresponding to $\mathbf{\Sigma}_\theta$.

\subsection{Hypertoric DM stacks and the GIT data}\label{GIT:hypertoricDMstack}

We give the GIT construction of hypertoric DM stacks.  The GIT data of the hypertoric DM stack is useful for the construction of the wall crossing.  We introduce the notion of {\em extended stacky hyperplane arrangements} and define the corresponding hypertoric DM stacks.  We prove that there is a one-to-one correspondence between the GIT data of the hypertoric DM stacks and the  extended stacky hyperplane arrangements, generalizing the idea in \S \ref{sec:LawrencetoricDMstack}.

\begin{defn}\label{extended:stacky:hyperplane:arrangement}
An {\em extended stacky hyperplane arrangement}
$\AA=(\bN,\beta,\theta, S)$ consists of the following data:
\begin{itemize}
\item $\bN$ is a finitely generated abelian group;
\item $\beta: \zz^m\to \bN$ is a map.  Also we write $b_i=\beta(e_i)$, the image of the standard generator $e_i$ of $\zz^m$;
\item $S\subset [m]$ and let $n:=m-|S|$, then $\beta_{\red}: \zz^n\to \bN$ determined by $\{b_i|i\in [m]\setminus S\}$ is a map. 
Let 
$$\beta_{\red}^\vee: (\zz^n)^\vee\to \ll_{\red}^{\vee}$$ be the Gale dual of $\beta_{\red}$. 
The element $\theta\in \ll_{\red}^\vee$ is a generic element in sense of \cite[\S 2]{JT}, and 
$\AA_{\red}=(\bN,\beta_{\red},\theta)$ is a stacky hyperplane arrangement in sense of \cite[Definition 2.1]{JT}. 
\item $\{b_i|i\in S\}$ are called the extended  vectors. 
\end{itemize}
\end{defn}


We prove that an extended stacky hyperplane arrangement $\AA$ determines an extended Lawrence stacky fan, hence a Lawrence toric DM stack $X_\theta$.   The hypertoric DM stack $Y_\theta:=Y_\AA$ is a closed substack of $X_\theta$.  First we write down some diagrams of exact sequences from the extended stacky hyperplane arrangements:
\[
\begin{CD}
0 @ >>>\mathbb{Z}^{n}@ >>> \mathbb{Z}^{m}@ >>>
\mathbb{Z}^{m-n} @
>>> 0\\
&& @VV{\beta_{\red}}V@VV{\beta}V@VV{}V \\
0@ >>> \bN @ >{\cong}>>\bN@ >>> 0 @>>> 0.
\end{CD}
\]
Taking Gale dual yields:
\[
\begin{CD}
0 @ <<<(\mathbb{Z}^{n})^\vee@ <<< (\mathbb{Z}^{m})^\vee@ <<<
(\mathbb{Z}^{m-n})^\vee @
<<< 0\\
&& @VV{\beta_{\red}^\vee}V@VV{\beta^\vee=(\beta_{\red}^\vee,\beta_S^\vee)}V@VV{=}V \\
0@ <<< \ll_{\red}^\vee @ <{}<<\ll^\vee@ <<< \zz^{m-n} @<<< 0.
\end{CD}
\]
Considering the following diagram:
\[
\begin{CD}
0 @ <<<(\mathbb{Z}^{n})^\vee\oplus (\mathbb{Z}^{n})^\vee@ <<< (\mathbb{Z}^{n})^\vee\oplus (\mathbb{Z}^{n})^\vee\oplus(\zz^{|S|})^\vee@ <<<
(\mathbb{Z}^{m-n})^\vee @
<<< 0\\
&& @VV{(\beta_{\red}^\vee,-\beta_{\red}^\vee)}V@VV{(\beta_{\red}^\vee,-\beta_{\red}^\vee, \beta_S^\vee)}V@VV{=}V \\
0@ <<< \ll_{\red}^\vee @ <{}<<\ll^\vee@ <<< \zz^{m-n} @<<< 0,
\end{CD}
\]
and taking Gale dual again:
\[
\begin{CD}
0 @ >>>\mathbb{Z}^{2n}@ >>> \mathbb{Z}^{2n+|S|}@ >>>
\mathbb{Z}^{m-n} @
>>> 0\\
&& @VV{\beta_{L, \red}}V@VV{\beta_{L}}V@VV{}V \\
0@ >>> \bN_{L} @ >{\cong}>>\bN_{L}@ >>> 0 @>>> 0.
\end{CD}
\]
For the map $\beta_{L,\red}: \zz^{2n}\to \bN_{L}$, and the generic element $\theta\in \ll_{\red}^{\vee}$, there is a Lawrence simplicial fan 
$\Sigma_{\theta}$ constructed in \cite[\S 2]{JT}.  Hence we have an extended Lawrence stacky fan 
$$\mathbf{\Sigma_{L}}=(\bN_{L}, \Sigma_{\theta}, \beta_{L}, S)$$
in Definition \ref{extended:Lawrence:stacky:fan}.  The Lawrence toric DM stack 
$$X_{\theta}:=X_{\mathbf{\Sigma_{L}}}=[U_\theta/K]$$
is defined in \S \ref{sec:LawrencetoricDMstack}.
Let $\cc[z_1,\cdots,z_n,w_1,\cdots,w_n,u_1,\cdots,u_{|S|}]$ be the coordinate ring of $\cc^{N}=\cc^{2n+|S|}$. 
Let $I_{\beta^{\vee}}$ be the ideal 
\begin{equation}\Label{ideal1}
I_{\beta^{\vee}}:=\left\langle\sum_{i=1}^{n}(\beta_{\red}^{\vee})^{*}(x)_{i}z_{i}w_{i}|~ x\in \ll_{\red}\right\rangle,
\end{equation}
where $(\beta_{\red}^{\vee})^{*}$ is the map $\ll_{\red}\to \zz^{n}$ and $(\beta_{\red}^{\vee})^{*}(x)_{i}$ is the $i$-th component
of the vector $(\beta_{\red}^{\vee})^{*}(x)$.
Let $V_{\theta}\subset U_{\theta}$ be the closed subvariety determined by the ideal in (\ref{ideal1}).  The hypertoric DM stack 
$$Y_\theta=[V_\theta/K]$$
is a quotient stack. 

\begin{prop}
The hypertoric DM stack $Y_\theta=Y_{\AA}$ is isomorphic to the hypertoric DM stack $Y_{\AA_{\red}}$ associated to its underlying stacky hyperplane arrangement $\AA_{\red}$ defined in \cite[\S 2]{JT}. 
\end{prop}
\begin{proof}
Any extended stacky hyperplane arrangement $\AA$ has an underlying stacky hyperplane arrangement 
$\AA_{\red}$ by forgetting the extra data $S$. 
The hypertoric DM stack is defined by a stacky hyperplane arrangement in \cite[\S 2]{JT}.   Since for extended stacky fans, the associated toric DM stack is isomorphic to the toric DM stack associated to its underlying stacky fan, see \cite{Jiang}, the hypertoric DM stack $Y_\theta=Y_{\AA}$ associated an extended hyperplane arrangement is also isomorphic to the hypertoric DM stack $Y_{\AA_{\red}}$ associated to its underlying stacky hyperplane arrangement.  The difference is that there are more extra power of $\Cstar$'s modulo by the same rank of power of $\Cstar$'s. We omit the details. 
\end{proof}

Hence the GIT data for hypertoric DM stack $Y_\theta$ consists of the following:
\begin{itemize} 
\item $K \cong (\Cstar)^r$, a connected torus of rank $r$; 
\item $\ll = \Hom(\Cstar,K)$, the cocharacter lattice of $K$; 
\item $D_1,\ldots,D_n,  -D_1, \cdots, -D_n, D_{1},\cdots, D_{|S|} \in \ll^\vee = 
\Hom(K,\Cstar)$, characters of $K$; 
\item stability condition $\theta\in  \ll^{\vee}\otimes\rr$.
\end{itemize} 

\begin{rmk}
From a similar argument as in \S \ref{sec:LawrencetoricDMstack}, given the GIT data of the hypertoric DM stack $Y_\theta$, we can construct an extended stacky hyperplane arrangement $\AA=(\bN,\beta,\theta, S)$, and vice versa. 
\end{rmk}

\subsection{The wall crossing of hypertoric DM stacks}
We prove that the single wall crossing of Lawrence toric DM stacks gives rise to the wall crossing of 
hypertoric DM stacks. 

Recall that the space $\ll^\vee \otimes \rr$ of stability conditions is divided into 
chambers by the closures of the sets $\angle_I$, $|I| = r-1$, 
and the Lawrence toric  DM stack $X_\theta$ depends on $\theta$ 
only via the chamber containing $\theta$.  
For any stability condition $\theta$, 
the set $U_\theta$ contains the big torus $T:=(\Cstar)^N$.
Thus for any two such stability conditions $\theta_1$,~$\theta_2$ 
there is a canonical birational map 
$X_{\theta_1} \dashrightarrow X_{\theta_2}$, 
induced by the identity transformation between 
$T/K \subset X_{\theta_1}$ and 
$T/K \subset X_{\theta_2}$.  

Let $C_+$,~$C_-$ be chambers 
in $\ll^\vee \otimes \rr$ that are separated by a hyperplane wall $W$, 
so that $W \cap \overline{C_+}$ is a facet of $\overline{C_+}$,
$W \cap \overline{C_-}$ a facet of $\overline{C_-}$, 
and $W \cap \overline{C_+} = W \cap \overline{C_-}$.  
Choose stability conditions $\theta_+ \in C_+$, $\theta_- \in C_-$ 
satisfying (A1-A2) and set $U_+ := U_{\theta_+}$, $U_- := U_{\theta_-}$,
$X_+ := X_{\theta_+}$, $X_- := X_{\theta_-}$, and
\begin{align*}
  & \A_\pm := \A_{\theta_{\pm}} = 
\big\{ I \subset \{1,2,\ldots,N\} : 
\theta_\pm \in \angle_I \big\} .
\end{align*}
Then $C_\pm = \bigcap_{I\in \A_\pm} \angle_I$. 
Let $\varphi \colon X_+ \dashrightarrow X_-$ be 
the birational transformation induced by the toric wall-crossing from $C_+$ to $C_-$.
Let $e \in \ll$ denote the \emph{primitive lattice vector} in $W^\perp$ 
such that $e$ is positive on $C_+$ and negative on $C_-$. 
We fix the notations
\begin{itemize} 
\item  $M_{+}:=\{i\in\{1,\cdots, N\}| D_i\cdot e>0\}$,
\item  $M_{-}:=\{i\in\{1,\cdots, N\}| D_i\cdot e<0\}$,
\item  $M_{0}:=\{i\in\{1,\cdots, N\}| D_i\cdot e=0\}$.
\end{itemize} 
From our construction of Lawrence toric DM stacks, $|M_{+}|=|M_{-}|$.

Choose $\theta_0$ from the relative interior of $W\cap \overline{C_+} 
= W \cap \overline{C_-}$. The stability condition $\theta_0$ 
does not satisfy (A1-A2) on the GIT data, but consider 
\begin{align*} 
\A_0 & := \A_{\theta_0} = 
\left\{ I \subset \{1,\dots,N\} : 
\theta_0 \in \angle_I \right\} 
\end{align*} 
and the corresponding toric Artin stack $X_0 := X_{\theta_0} =[U_{\theta_0}/K]$.
Here $X_0$ is not Deligne--Mumford, as 
the $\Cstar$-subgroup of $K$ 
corresponding to $e\in \ll$ (the defining equation 
of the wall $W$) has a fixed point in $U_0 := U_{\theta_0}$. 
The stack $X_0$ contains both $X_+$ and $X_-$ as open substacks.

The canonical line bundles of $X_{+}$ and $X_-$ 
are 
given by the character $-\sum_{i=1}^{2n} D_i=0$ of $K$. 
This means that Lawrence toric DM stacks are Calabi-Yau. 
There are canonical blow-down maps $g_\pm \colon X_\pm \to \overline{X}_0$, 
and $K_{X_\pm}=g_\pm^\star \oO_{\overline{X}_0}$.  We have a commutative diagram:
\begin{equation}\label{eq:K-equivalence}
\xymatrix{
  &~\tX  \ar[rd]^{f_-} \ar[ld]_{f_+} & \\ 
  X_+ \ar[rd]_{g_+} \ar@{-->}^{\varphi}[rr] &  & X_- \ar[ld]^{g_-} \\ 
  & \overline{X}_0 & 
}
\end{equation}
So the 
birational map $\varphi$ is \emph{crepant}, since $f_+^\star(K_{X_+}) = f_-^\star(K_{X_-})$ are trivial.

To construct $\tX$, consider the action of $K \times \Cstar$ on $\cc^{N+1}$ defined by the characters 
$\tD_1,\ldots,\tD_{N+1}$ of $K \times \Cstar$, where:
\[
\tD_j = 
\begin{cases}
  D_j \oplus 0 & \text{if $j < N+1$ and $D_j \cdot e \leq 0$} \\
  D_j \oplus ({-D_j} \cdot e) & \text{if $j < N+1$ and $D_j \cdot e > 0$} \\
  0 \oplus 1 & \text{if $j=N+1$}
\end{cases}
\]
Consider the chambers $\widetilde{C}_+$,~$\widetilde{C}_-$, 
and~$\widetilde{C}$ in $(\ll \oplus \zz)^\vee \otimes \rr$ 
that contain, respectively, the stability conditions
\begin{align*}
  \widetilde{\theta}_+ = (\theta_+,1) &&
  \widetilde{\theta}_- = (\theta_-,1) && \text{and} && 
  \widetilde{\theta} = (\theta_0, - \varepsilon)
\end{align*}
where $\varepsilon$ is a very small positive real number.  
Let $\tX$ denote the toric DM stack 
defined by the stability condition $\widetilde{\theta}$.  
We have, by \cite[Lemma~6.16]{CIJ}, that
the toric DM stack corresponding 
to the chamber $\widetilde{C}_{\pm}$ is $X_{\pm}$.
Furthermore, there is a commutative diagram as in \eqref{eq:K-equivalence}, 
where:
   $f_{\pm} \colon \tX \to X_{\pm}$ are toric blow-ups, 
   arising from the wall-crossing from $\widetilde{C}$ to $\widetilde{C}_{\pm}$.

Now we take into account the ideal (\ref{ideal1}).  
We have the corresponding hypertoric DM stacks 
$$Y_{\pm}:=Y_{\theta_{\pm}}\subset X_{\pm}$$
and 
the hypertoric stack
$Y_0\subset X_0$
determined by the ideal $I_{\beta^\vee}$ in (\ref{ideal1}).
Let $\overline{Y}_0$ be the coarse moduli space of $Y_0$ and let 
$\tY\subset \tX$
be the substack determined by the ideal $I_{\beta^\vee}$ in (\ref{ideal1}).

\begin{prop}
We have the following diagram:
\begin{equation}\label{eq:K-equivalence:hypertoricDMstack}
\xymatrix{
  &~\tY  \ar[rd]^{F_-} \ar[ld]_{F_+} & \\ 
  Y_+ \ar[rd]_{G_+} \ar@{-->}^{\phi}[rr] &  & Y_- \ar[ld]^{G_-} \\ 
  & \overline{Y}_0 & 
}
\end{equation}
which gives the crepant transformation morphism of hypertoric DM stacks. 
\end{prop}
\begin{proof}
The contractions $G_+, G_-$ are constructed in \cite[\S 4]{JT3}.  The maps $F_\pm$ are induced from the maps 
$f_\pm$ in (\ref{eq:K-equivalence}).  The birational map $\phi$ is crepant since $Y_\pm$ are Calabi-Yau stacks. 
\end{proof}

\begin{example}
We consider the following GIT data:
\begin{itemize}
\item $K\cong \Cstar$;
\item  $D_1=(1), D_2=(2), D_3(=(-1), D_4=(-2)\in \ll^\vee=\Hom(K,\Cstar)=\zz$;
\item stability conditions $\theta_+=(1), \theta_-=(-1)\in \ll^\vee\otimes\rr$.
\end{itemize}
We can easily construct the Lawrence toric DM stacks 
$$X_{\pm}=\oO_{\pp(1,2)}(-1)\oplus\oO_{\pp(1,2)}(-2).$$
The crepant transformation $\varphi: X_+\dasharrow X_-$ is an Atiyah type flop. 

We have the stacky hyperplane arrangements 
$\AA_\pm=(\bN, \beta_\pm, \theta_\pm)$, where 
$\beta_{\pm}: \zz^2\to \bN$ is the Gale dual of $\beta_{\pm}^\vee: \zz^2\to \ll^\vee$ determined by 
$\{D_1, D_2, D_3, D_4\}$. The corresponding hypertoric DM stacks 
$$Y_\pm=T^*\pp(1,2).$$
The crepant birational map $\phi: Y_+\dasharrow Y_-$ is a Mukai type flop. 
\end{example}

\section{The Fourier-Mukai transformation}\label{Fourier-Mukai}

\subsection{The $\bbT:=T\times\Cstar$-action on $X_\pm$ and $Y_\pm$}

From the construction of Lawrence toric DM stacks $X_\pm$ and hypertoric DM stacks $Y_\pm$ in \S \ref{wall:crossing}, there is a torus 
$\bbT:=T\times\Cstar$ action, where $T=(\Cstar)^m$.   Since 
$U_{\pm}\subset \cc^{2n}\times\cc^{|S|}$,  the action $T$ acts on $U_\pm$ by the standard action which is given  by
\begin{align*}
&(\lambda_1,\cdots,\lambda_n,\lambda_{n+1}, \cdots,\lambda_m)(z_1,\cdots, z_n, w_1,\cdots, w_n, u_1,\cdots, u_{|S|})\\
&=
(\lambda_1z_1,\cdots,\lambda_nz_n,\lambda_1^{-1}w_1,\cdots,\lambda_n^{-1}w_n, \lambda_{n+1}u_1, \cdots,\lambda_mu_{|S|});
\end{align*}
the extra $\Cstar$ acts by scaling the fibre of $T^*\cc^n$. We consider the $\bbT$-action on $\tX$ induced from the inclusion 
$\bbT=(\Cstar)^{m+1}\times\{1\}\subset \bbT\times\Cstar$ and the $\bbT\times\Cstar$ action on $\tX$. 

The $\bbT$-fixed points on $X_\pm$ and $Y_\pm$ are isolated, and in one-to-one correspondence with the minimal anticones 
$\delta_{\pm}\in\A_{\pm}$.  Using the correspondence between anticones and cones in the Lawrence toric fan 
$\Sigma_{\pm}$, the torus fixed points are given by top dimensional cones in the fan.  
The torus fixed points all lie in the core of $X_\pm$ and $Y_\pm$, which are the closed projective substacks inside $X_\pm$ and 
$Y_\pm$.

\subsection{The $\bbT$-equivariant K-theory of $X_\pm$, $Y_\pm$ and the Fourier-Mukai transformation}

\subsubsection{Equivariant K-theory}

Let $K_{\bbT}^0(X_\pm)$, $K_{\bbT}^0(Y_\pm)$ be the $\bbT$-equivariant Grothendieck K-groups of coherent sheaves. 
They are modules over $K_{\bbT}^0(\pt)$, the ring $\zz[\bbT]$ of regular functions on the torus $\bbT$. 

For the Lawrence toric DM stack $X_\theta$ for a stability condition $\theta\in\ll^\vee\otimes\rr$, the 
$\bbT$-equivariant divisor $\{z_i=0\}$, $\{w_i=0\}$ or $\{u_j\}=0$ for $1\leq i\leq n$, $1\leq j\leq |S|$ on $X_\theta$
determine $\bbT$-equivariant line bundles
$R_i$ over $X_\theta$ for $1\leq i\leq N$.  We denote by $R_i$ the equivalent classes of such line bundles.  Similar construction works for the toric DM stack $\tX$. 
We fix notation for $\bbT$-equivariant line bundles for $X_\pm$ and $\tX$. 
$$R_1^\pm,\cdots, R_{N}^\pm\in K_\bbT^0(X_\pm),$$ and 
$$\widetilde{R}_1,\cdots, \widetilde{R}_{N}, \widetilde{R}_{N+1}\in K_\bbT^0(\tX).$$
Fix notations of their inverses:
$$S_j^+:=(R_j^+)^{-1}, \quad\quad S_j^-:=(R_j^-)^{-1}, \quad\quad  \widetilde{S}_j=\widetilde{R}_j^{-1}.$$
Let $\hbar$ be the $\bbT$-equivariant line bundle over $X_\pm$ defined by the extra $\Cstar$ in 
$\bbT=T\times\Cstar$. 

Let $X_\theta$ be a Lawrence toric DM stack corresponding to a stability condition $\theta$. 
Each character $p\in \Hom(K,\Cstar)=\ll^\vee$ defines a line bundle 
$L(p)\to X_\theta$: 
$$L(p)=U_\theta\times\cc/(z, v)\simeq (g\cdot z, p(g)\cdot v), g\in K.$$ 
The line bundle $L(p)$ is equipped with the $\bbT$-action $[z,v]\mapsto [t\cdot z, v], t\in \bbT$ as in  \cite[\S 6.3.2]{CIJ}. So it defines an element in 
$K_{\bbT}^0(X_\theta)$. The line bundles $R_i^\pm$ is:
$$R_i^\pm=L_{\pm}(D_i)\otimes e^{\lambda_i}, 1\leq i\leq n,$$
$$R_i^\pm=L_{\pm}(D_i)\otimes e^{\lambda-\lambda_i}, n+1\leq i\leq 2n,$$
and 
$$R_i^\pm=L_{\pm}(D_i)\otimes e^{\lambda_i}, 2n+1\leq i\leq N.$$
where $e^{\lambda_i}\in\cc[\bbT]$ stands for the irreducible $\bbT$-representation given by the $i$-th component $\bbT\to \Cstar$.  Since $\bbT=(\Cstar)^m\times\Cstar=(\Cstar)^{m+1}$, we use $\lambda$ to represent the $m+1$-equivariant parameter 
$\lambda_{m+1}$. 

For $(p,n)\in \ll^\vee\oplus\zz$, the $\bbT$-equivariant line bundle $L(p,n)\to \tX$ is similarly defined, and we have:
$$\widetilde{R}_i=L(\widetilde{D}_i)\otimes e^{\lambda_i}, 1\leq i\leq N$$
as above, and 
$$\widetilde{R}_{m+1}=L(\widetilde{D}_{N+1})=L(0,1).$$
As in \cite[\S 6.3.2]{CIJ}, the classes $L_{\pm}(p), (p\in \ll^\vee)$ generate the equivariant K-group $K_{\bbT}^0(X_\pm)$, and the classes $\{L(p,n)| (p,n)\in\ll^\vee\oplus\zz\}$ generate the equivariant K-group $K_\bbT^0(\tX)$.

\subsubsection{Localized K-theory basis}

Let $\delta_-\in\A_-$ be a minimal anticone,  $x_{\delta_-}$ be the $\bbT$-fixed point on $X_-$, and 
$$i_{\delta_-}: x_{\delta_-}\to X_-$$
be the inclusion.  Denote by $G_{\delta_-}$ the isotropy group of $x_{\delta_-}$.  Then 
$x_{\delta_-}=BG_{\delta_-}$.  From \cite[\S 6.3.2]{CIJ}, 
$$i_{\delta_-}^\star R_i=1,  \text{~for~} i\in\delta_-.$$
A localized basis of $K_\bbT^0(X_-)$, after inverting the non-zero elements of $\zz[\bbT]$, is given by:
\begin{equation}\label{K-theory:basisX-}
\{(i_{\delta_-})_\star\varrho: \varrho \text{~an irreducible representation of~} G_{\delta_-}, \delta_-\in\A_- \text{~a minimal anticone}\}.
\end{equation}
By Koszul complex the structure sheaf $\oO_{x_{\delta_-}}$ is given by 
$$e_{\delta_-}:=\prod_{i\notin\delta_-}(1-S_i^-).$$
Since the irreducible representation of $G_{\delta_-}$ is given by the $\bbT$-linearization on $(i_{\delta_-})_\star\varrho$. Choosing a lift 
$\hat{\varrho}\in\Hom(K,\Cstar)=\ll^\vee$ for each $G_{\delta_-}$-representation $\varrho: G_{\delta_-}\to\Cstar$, then 
$$e_{\delta_-,\varrho}:=L_-(\hat{\varrho})\cdot \prod_{i\notin\delta_-}(1-S_i^-).$$
Similarly, $\{e_{\delta_+,\varrho}\}$ is a basis for the localized $\bbT$-equivariant K-theory $K_\bbT^0(X_+)$.

\subsubsection{Fourier-Mukai transformation}

In \cite[\S 6.3.3]{CIJ}, the authors calculate the Fourier-Mukai transformation for the localized $\bbT$-equivariant K-theory 
$K_\pm^0(X_\pm)$.

\begin{prop}[\cite{CIJ}, Theorem 6.19] \label{FM:LawrenceDMstack}
\hfill
\begin{enumerate}
\item If $\delta_-\in\A_-$ is a minimal anticone such that $\delta_-\in\A_+$, then 
$$\FM(e_{\delta_-,\rho})=e_{\delta_+,\rho},$$
where $\delta_+=\delta_-$ is the same anticone, but taken as in $\A_+$;
\item If $\delta_-\in\A_-$, but $\delta_-\notin\A_+$, then 
$\FM(e_{\delta_{-},\varrho})$ is equal to 
  \[ 
  \frac{1}{l} 
  \sum_{t\in \tT}
  \left(
    \frac{1-S_{j_{-}}^+}{1-t^{-1} }
    \cdot 
    L_+(\hat{\varrho}) t^{\hat{\varrho} \cdot e} 
    \cdot 
    \prod_{\substack{j\notin \delta_{-}\\ D_j \cdot e <0}} 
    (1-S_j^+)
    \cdot 
    \prod_{\substack{i\notin \delta_{-} \\ D_i\cdot e\ge 0}}  
   \big(1-t^{-D_i \cdot e}S_i^+\big)
   \right)
  \]
where $j_-$ is the unique element of $\delta_-$ such 
that $D_{j_-} \cdot e<0$, $l = -D_{j_-} \cdot e$ and 
  \[
  \tT =\left\{\zeta \cdot (R_{j_{-}}^+)^{1/l} :  
\zeta \in\mu_l \right\}. 
  \]
\end{enumerate}
\end{prop}

\subsection{The $\bbT$-equivariant K-theory $K_\bbT^0(Y_{\pm})$ of hypertoric DM stacks}

Let $$\iota_{\pm}: Y_\pm\hookrightarrow X_\pm$$
be the inclusion of hypertoric DM stacks 
$Y_\pm$ to their associated Lawrence toric DM stacks $X_\pm$. 

\begin{lem}\label{pullback:isomorphism:KtheoryY}
The pullback 
$$\iota_{\pm}^\star: K_{\bbT}^0(X_\pm)\stackrel{\cong}{\rightarrow} K_{\bbT}^0(Y_\pm)$$
is an isomorphism on the $\bbT$-equivariant K-theory. 
\end{lem}
\begin{proof}
This is from \cite[Theorem 5.4]{Edidin}.  On the other hand, one can directly calculate that the equivariant 
K-theory groups $K_{\bbT}^0(Y_\pm)$ are also generated by line bundles 
$R_1^\pm, \cdots, R_N^\pm$ modulo the same relations as in the Lawrence toric DM stack case. 
\end{proof}
\begin{rmk}
Let $\hbar$ be the $\bbT$-equivariant line bundle over $X_\pm$ given by the extra factor $\Cstar$ in $\bbT=T\times\Cstar$. 
For the Lawrence toric DM stacks $X_\pm$, 
$$K_{\bbT}^0(X_\pm)\cong \frac{\cc[(R_1^{\pm})^{\pm 1},\cdots, (R_N^{\pm})^{\pm 1}, \hbar^{\pm 1}]}{I+J}, $$
where 
$$I=\{R_{i_1}^\pm\cdots R_{i_k}^\pm| \overline{\{i_1,\cdots,i_k\}}\notin\A_\pm\}$$
is the ideal generated by 
the products for subsets $\{i_1,\cdots,i_k\}$ not lying in the set of anticones; and 
$$J=\{R_i^{\pm}-(\hbar\cdot (R_{n+i}^\pm)^{-1})| 1\leq i\leq n\}$$  is the ideal 
generated by the relation of the line bundles $R_i^\pm$ and $R_{n+i}^\pm$. 
\end{rmk}

To study the Fourier-Mukai transformation of the crepant birational map $\phi: Y_+\dasharrow Y_-$, we set up the following diagram:
\begin{equation}\label{diagram:XY}
\xymatrix{
&\tY\ar[dl]_{F_+}\ar[dr]^{F_-}\ar@{^{(}->}[d]^{\tilde{\iota}}&\\
Y_+\ar@{^{(}->}[d]_{\iota_+}&\tX\ar[dl]_{f_+}\ar[dr]^{f_-}& Y_-\ar@{^{(}->}[d]^{\iota_-}\\
X_+& & X_-
}
\end{equation}

We denote by $\Phi:=\FM$ the equivariant Fourier-Mukai transformation for $X_\pm$. 
The Fourier-Mukai transformation 
$$\Psi: K_{\bbT}^0(Y_-)\to K_{\bbT}^0(Y_+)$$
is given by:
$$E\mapsto \Psi(F)=(F_+)_{\star}F_{-}^{\star}(E).$$
\begin{prop}\label{commutative:diagram:K:XY}
There is a commutative diagram of $\bbT$-equivariant K-theory groups:
\[
\xymatrix{
K_{\bbT}^0(X_-)\ar[r]^{\Phi}\ar[d]_{\iota_{-}^\star}  & K_{\bbT}^0(X_+) \ar@{->}[d]^{\iota_{+}^\star} \\
K_{\bbT}^0(Y_-)\ar[r]^{\Psi}&K_{\bbT}^0(Y_+) 
}
\]
which implies that $\Psi$ is an isomorphism on K-theory groups. 
\end{prop}
\begin{proof}
This is similar to \cite[Lemma 7.7]{CIJ}.  The difference here is that by Proposition \ref{pullback:isomorphism:KtheoryY}, 
the pullbacks $\iota_{\pm}^{\star}$ are isomorphisms. Then we directly check the commutative diagram using (\ref{diagram:XY}):
for any element $E\in K_{\bbT}^0(X_-)$, 
\begin{align*}
\iota_{+}^\star\circ\Phi(E)&=\iota_{+}^\star\circ((f_+)_{\star}f_{-}^\star(E))\\
&=(F_{+})_{\star}\circ \tilde{\iota}^{\star}\circ f_{-}^\star(E)\\
&=(F_{+})_{\star}\circ F_{-}^\star\circ \iota_{-}^\star(E)\\
&=\Psi\circ \iota_{-}^\star(E).
\end{align*}
\end{proof}

The torus $\bbT$-fixed points of $Y_\pm$ are the same as the torus $\bbT$-fixed points of $X_\pm$, which all lie in the core. 
The localized K-theory basis of $K_{\bbT}^0(Y_{\pm})$ are also generated by the minimal anticones $\delta_{\pm}\in\A_\pm$. 
Let $\delta_-\in\A_-$ be a minimal anticone. Then 
$$i_{\delta_-}: x_{\delta_-}\hookrightarrow Y_-\subset X_-$$
is the inclusion of the fixed point $x_{\delta_-}$.  For each $p\in\ll^\vee$, it also defines a line bundle 
$$L_-^{Y_-}(p)=V_{-}\times\cc/K,$$
which is the pullback $\iota_{-}^\star L_-(p)$ of $L_-(p)$ on $X_-$.  From now on we denote by $L_-(p)$ the line bundle on $Y_-$ determined by $p\in\ll^\vee$. 

Set 
$$e_{\delta_{-},\varrho}^{Y_-}:=L_-(\hat{\varrho})\cdot \prod_{i\notin \delta_-}(1-S_i),$$
where $\hat{\varrho}\in\ll^\vee$ is the lift of $\varrho$.  Then $\{e_{\delta_{-},\varrho}^{Y_-}\}$ is a basis for the localized 
$\bbT$-equivariant K-theory of $Y_-$.  Similarly we have a localized 
$\bbT$-equivariant K-theory basis $\{e_{\delta_{+},\varrho}^{Y_+}\}$ of $Y_+$.

\begin{prop}[Fourier-Mukai transformation for hypertoric DM stacks]\label{FM:hypertoricDMstack}
\hfill
\begin{enumerate}
\item If $\delta_-\in\A_-$ is a minimal anticone such that $\delta_-\in\A_+$, then 
$$\Psi(e_{\delta_-,\rho})=e_{\delta_+,\rho},$$
where $\delta_+=\delta_-$ is the same anticone, but taken as in $\A_+$;
\item If $\delta_-\in\A_-$, but $\delta_-\notin\A_+$, then 
$\Psi(e_{\delta_{-},\varrho})$ is equal to 
  \[ 
  \frac{1}{l} 
  \sum_{t\in \tT}
  \left(
    \frac{1-S_{j_{-}}^+}{1-t^{-1} }
    \cdot 
    L_+(\hat{\varrho}) t^{\hat{\varrho} \cdot e} 
    \cdot 
    \prod_{\substack{j\notin \delta_{-}\\ D_j \cdot e <0}} 
    (1-S_j^+)
    \cdot 
    \prod_{\substack{i\notin \delta_{-} \\ D_i\cdot e\ge 0}}  
   \big(1-t^{-D_i \cdot e}S_i^+\big)
   \right)
  \]
where $j_-$ is the unique element of $\delta_-$ such 
that $D_{j_-} \cdot e<0$, $l = -D_{j_-} \cdot e$ and 
  \[
  \tT =\left\{\zeta \cdot (R_{j_{-}}^+)^{1/l} :  
\zeta \in\mu_l \right\}. 
  \]
\end{enumerate}
\end{prop}
\begin{proof}
This result is from Proposition \ref{commutative:diagram:K:XY}, and the Fourier-Mukai transformation formula 
$\Phi$ in Proposition \ref{FM:LawrenceDMstack}. 
\end{proof}

\section{Analytic continuation of the quantum connection}\label{analytic:continuation}

In this section we prove that the Fourier-Mukai transformation $\Psi: K_{\bbT}^0(Y_-)\to K_{\bbT}^0(Y_+)$ matches the analytic continuation of quantum connections for $Y_\pm$, hence the monodromy conjecture. 

\subsection{Equivariant quantum cohomology}
This section serves as a general introduction to equivariant quantum cohomology.  We fix a smooth DM stack $X$ with the torus 
$\bbT$-action. 
\subsubsection{The $\bbT$-equivariant quantum cohomology}
The moduli stack 
$\cM_{0,n}(X, d)$ of degree $d\in H_2(X, \qq)$ twisted stable maps to $X$ carries a $\bbT$-action, and a 
virtual fundamental cycle $[\cM_{0,n}(X,d)]^{\virt}\in H_{*, \bbT}(\cM_{0,n}(X,d))$.
There are $\bbT$-equivariant evaluation maps\footnote{We ignore the issue of trivializing the marked gerbes in our moduli problem. A detailed discussion on this can be found in \cite{AGV}.}:
$$\ev_i: \cM_{0,n}(X,d)\to IX$$
to the inertia stack $IX$ of $X$ for $1\leq i\leq n$, see \cite{CR2}, \cite{AGV}. 

Given $\gamma_1,...,\gamma_n\in H_{\CR, \bbT}^*(X)$, we consider the following genus $0$ $\bbT$-equivariant Gromov-Witten invariant:
$$\langle \gamma_1,\cdots,\gamma_n\rangle_{0,n,d}^{X}=\int^\bbT_{[\cM_{0,n}(X,d)]^{\virt}}\prod_{i}\ev_i^\star\gamma_i.$$

The moduli stack $\cM_{0,n}(X,d)$ has components indexed by the components of the inertia stack $IX$. 
We write 
$$IX=\bigsqcup_{f\in \sf B}X_f$$
for the decomposition of $IX$ into connected components, where $\sf B$ is the index set. 
Then the component $\cM_{0,n}(X,d)^{f_1,\cdots,f_n}$ is the one which under evaluation maps $\ev_i$, the images lie in
the component $X_{f_i}$. 
The virtual dimension of $\cM_{0,n}(X,d)^{f_1,\cdots,f_n}$ is:
\begin{equation}\label{virtual:dimension}
-K_{X}\cdot d+\dim(X)+n-3-\sum_{i}\age(X_{f_i}).
\end{equation}

If $X$ is not compact (like our Lawrence and hypertoric DM stacks), then the moduli stack $\cM_{0,n}(X, d)$ is non-compact. 
There is a $\bbT$-action on $\cM_{0,n}(X, d)$.  Assume that the $\bbT$-fixed locus $\cM_{0,n}(X, d)^\bbT$ is compact, then $\bbT$-equivariant GW invariants can be defined in the same way, replacing equivariant integration by equivariant residues.


Let $\NE(X)\subset H_2(X, \rr)$ be the cone generated by classes of effective curves and set 
$$\NE(X)_{\zz}:=\{d\in H_2(X,\zz): d\in \NE(X)\}.$$
Let $R_\bbT:=H^*_{\bbT}(\pt)$ and $R_\bbT[\![Q]\!]$ the formal power series ring
\[
R_\bbT[\![Q]\!] = \left\{ \sum_{d \in \tiny\NE(X)_\zz} a_d Q^d : a_d \in R \right\}
\]
so that $Q$ is a so-called \emph{Novikov variable} (see e.g. \cite[III~5.2.1]{Manin}).  For $\gamma_i, \gamma_j, t\in H_{\CR,\bbT}^*(X)$, the big $\bbT$-equivariant quantum product is defined by:
\begin{equation}\label{big:quantum:product}
(\gamma_i\star_{t} \gamma_j, \gamma_k)=\sum_{d\in \tiny\NE(X)_{\zz}}\sum_{n\geq 0}Q^d\langle \gamma_i, \gamma_j, \underbrace{t,\cdots,t}_{n}, \gamma_k\rangle_{0,n+3,d}^{X}
\end{equation}
The small $\bbT$-equivariant quantum product is defined by putting $n=0$:
\begin{equation}\label{quantum:product2}
(\gamma_i\star_{\sm} \gamma_j, \gamma_k)=\sum_{d\in \tiny\NE(X)_{\zz}}Q^d\langle \gamma_i, \gamma_j, \gamma_k\rangle_{0,3,d}^{X}
\end{equation}
or 
\[
\gamma_i\star_{\sm} \gamma_j=\sum_{d\in \tiny\NE(X)_{\zz}}Q^d\cdot \inv^\star\cdot \ev_{3,\star}(\ev_1^\star(\gamma_i)\ev_2^\star(\gamma_j)\cap [\cM_{0,3}(X,d)]^{\virt})
\]
where $\inv: IX\to IX$ denotes the involution sending $(x, g)\mapsto (x, g^{-1})$, for 
$x\in X, g\in \Aut(x)$. 
The big quantum product satisfies the associativity property and makes $H_{\CR,\bbT}^*(X)\otimes R_\bbT[\![Q]\!]$ a ring, which is called the  equivariant  quantum cohomology ring. 

We briefly review the Givental's formalism about the orbifold Gromov-Witten invariants in terms of the Lagrangian cone in certain symplectic vector space, see \cite{CIJ}. Let 
\[
\mathcal{H}(X):=H^*_{\CR,\bbT}
(X,\mathbb{C})\otimes_{R_{\bbT}}S_{\bbT}(\!(z^{-1})\!)[\![Q]\!],
\]
equipped the non-degenerate $S_{\bbT}[\![Q]\!]$-bilinear  symplectic form
\[
\Omega(f,g):=\mbox{Res}_{z=0}(f(-z),g(z))_{\CR}dz,
\]
where $(-,-)_{\CR}$ is the orbifold Poincar\'e pairing. 
Here $S_{\bbT}$ is the localization ring of $R_{\bbT}$ with respect to the multiplicative set of nonzero homogeneous elements in $R_{\bbT}$. 
Let 
$$\mathcal{H}_{+}:=
H^*_{\CR,\bbT}(X)\otimes_{R_{\bbT}}S_{\bbT}[z][\![Q]\!]; \quad 
\mathcal{H}_{-}:=
z^{-1}H^*_{\CR,\bbT}(X)\otimes_{R_{\bbT}}S_{\bbT}[z^{-1}][\![Q]\!].$$
Then $\mathcal{H}(X)=\mathcal{H}_{+}\oplus \mathcal{H}_{-}$ and one can think of $\mathcal{H}(X)=T^*(\mathcal{H}_{+})$.
The genus zero descendant Gromov-Witten potential is a formal function
$\sF_{X}^0: (\sH_+, -z)\to S_{\bbT}[\![Q]\!]$ defined on the formal neighbourhood of $-z$ in 
$\sH_+$ and taking values in $S_{\bbT}[\![Q]\!]$:
$$\sF_{X}^0(-z1+\mathbf{t}(z))=\sum_{d\in\tiny\NE(X)_{\zz}}\sum_{n=0}^{\infty}\frac{Q^d}{n!}\langle\mathbf{t}(\psi),\cdots,\mathbf{t}(\psi)\rangle_{0,n,d}^{X}.$$
Here $\mathbf{t}(z)=\sum_{n=0}^{\infty}t_nz^n$ with $t_n\in H_{\CR,\bbT}^*(X)\otimes_{R_{\bbT}}S_{\bbT}[\![Q]\!]$.

Givental's Lagrangian cone $\mathcal{L}_{X}$ is the graph of the differential 
$\mathcal{F}_{\mathcal{X}}^0$,
more explicitly, 
\[
\mathcal{L}_{X}:=
\{
(p,q)\in \mathcal{H}_{-}\oplus\mathcal{H}_{+}|p=d_{q}\mathcal{F}^{0}_{X}
\}\subset\mathcal{H}.
\]
Tautological equations for genus $0$ Gromov-Witten invariants imply that $\mathcal{L}_X$ is a cone ruled by a {\em finite dimensional} family of affine subspaces. A particularly important finite-dimensional slice of $\mathcal{L}_X$ is the {\em $J$-function}:
$$J_X(t,z)=1z+t+\sum_{n,  d}\sum_\alpha \frac{Q^d}{n!}\langle t,...,t, \frac{\phi_\alpha}{z-\bar{\psi}}\rangle_{0,n+1,d}\phi^\alpha,$$
where $\{\phi_\alpha\}, \{\phi^\alpha\}\subset H_{\CR,\bbT}^*(X)$ are additive bases dual to each other under $(-,-)_{\CR}$.

\subsubsection{Quantum connection}

We fix a homogeneous basis 
$$\phi_0,\cdots, \phi_{R}$$
for the $\bbT$-equivariant Chen-Ruan cohomology 
$H_{\CR,\bbT}^*(X_\pm)\cong H_{\CR,\bbT}^*(Y_\pm)$.  Let 
$$\tau^0,\cdots, \tau^R$$
be the corresponding dual co-ordinates.  The equivariant quantum connection is the operator
$$\bigtriangledown_i:  H_{\CR,\bbT}^*(X)[z]\otimes R_\bbT[\![Q]\!][\![\tau^0,\cdots,\tau^R]\!]\to 
H_{\CR,\bbT}^*(X)[z]\otimes R_\bbT[\![Q]\!][\![\tau^0,\cdots,\tau^R]\!]$$
defined by:
$$\bigtriangledown_i=\frac{\partial}{\partial \tau^i}+\frac{1}{z}(\phi_i\star -),$$
where $\phi\star -$ stands for the big quantum product. 

We return to the Lawrence toric DM stacks $X_\pm$, and the hypertoric DM stacks $Y_\pm$. By \cite[Proposition 6.4]{JT3}, the pullback $\iota_\pm^\star$ equate GW invariants of $X_\pm$ and $Y_\pm$. It follows that for $\phi\in H_{\CR,\bbT}^*(X_\pm)$, we have an equality of operators:
$$\phi\star_t=\iota_\pm^\star(\phi)\star_{\iota_\pm^\star t}.$$
So 
$$\iota_{\pm}^\star\bigtriangledown_i^{X_\pm}=\bigtriangledown_i^{Y_\pm}.$$

\subsection{The Analytic continuation}

\subsubsection{I-functions of $X_\pm$ and $Y_\pm$}\label{I-functionXpm}

Recall that in \cite[\S 5.3]{CIJ}, the global extended K\"ahler moduli space is a universal cover of space 
$\M$, where $\M$ is the toric variety of the GKZ-fan on $\ll^\vee\otimes\rr$. 
The GKZ-fan is given by the chamber structures on $\ll^\vee\otimes\rr$.  The rank 
$\rk(\ll^\vee)=r$.

Our wall $W=\overline{C}_{+}\cap\overline{C}_{-}$, where $C_{\pm}$ are cones of $\ll^\vee\otimes\rr$. 
The two torus fixed points $P_+$ and $P_-$ corresponding to $C_+$ and $C_-$ are called the large radius limit points.
The two toric DM stacks $X_\pm$ are called the mirrors corresponding to these two points.  The hypertoric DM stacks $Y_\pm$ have the same LG-mirrors as $X_\pm$, and the toric variety $\M$ is also the global extended K\"ahler moduli space corresponding to $Y_\pm$. 

Let 
$$\ell_{\pm}:=\dim(H^2(X_\pm,\rr))\cong \dim(H^2(Y_\pm,\rr))=r-\#(S_\pm).$$
Recall that the wall $W$ has rank $r-1$.    
The lattice $\ll^\vee$ has the bases for both $X_\pm$ and $Y_\pm$. 
We order the bases
$$\{p_1^+,\cdots,p_{\ell_{+}}^+\}\cup\{D_j:j\in S_+\}=\{\sfp_1^+,\cdots, \sfp_{r-1}^+, \sfp_r^+\}$$
$$\{p_1^-,\cdots,p_{\ell_{-}}^-\}\cup\{D_j:j\in S_-\}=\{\sfp_1^-,\cdots, \sfp_{r-1}^-, \sfp_r^-\}$$
in such a way that $\sfp_i^+=\sfp_i^-\in W$ for $i=1,\cdots,r-1$ and
$\sfp_r^{\pm}$ be the unique vector that does not lie on the wall $W$. Let 
$$\{y_i, x_j: 1\leq i\leq \ell_{+}, j\in S_{+}\}=\{\sfy_1,\cdots, \sfy_r\}$$
$$\{\tilde{y}_i, \tilde{x}_j: 1\leq i\leq \ell_{+}, j\in S_{-}\}=\{\tilde{\sf{y}}_1,\cdots, \tilde{\sf{y}}_r\}$$
be the corresponding reordering coordinates of $\M$. 
Then 
\[
\tilde{\sf{y}}_i=
\begin{cases}
\sfy_i\cdot \sfy_{r}^{c_i}, & 1\leq i\leq r-1;\\
\sfy_{r}^{-c}, & i=r
\end{cases}
\]
for some $c_i\in\qq, c\in \qq_{>0}$. 

For $d\in\ll\otimes\qq$, we write 
$$d=\bar{d}+\sum_{j\in S_\pm}(D_j.d)\xi_j$$
where $\bar{d}$ is the $H_2(X_\pm,\rr)$-component of $d$.  The $H_{\CR,\bbT}^*(X_\pm)$-valued hypergeometric series 
$I_\pm(\sfy, z)\in H_{\CR,\bbT}^*(X_\pm)\otimes_{R_{\bbT}}R_{\bbT}(\!(z^{-1})\!)[\![Q,\sigma_\pm,x]\!]$ is:
\begin{equation}
  \label{eq:I+}
  I_\pm(\sfy,z):=
e^{\sigma_\pm/z} 
  \sum_{d \in \kk_\pm} 
  \sfy^d 
  \left(
    \prod_{j=1}^{N}
    \frac{\prod_{a : \<a\rangle=\< D_j\cdot d \rangle, a \leq 0}(u_j+a z)}
    {\prod_{a : \<a \rangle=\< D_j \cdot d \rangle, a\leq D_j \cdot d} (u_j+a z)} 
  \right) 
  \mathds{1}_{[{-d}]}
\end{equation}
where 
\begin{itemize}
\item $\kk_\pm=\{f\in\ll\otimes\qq: \{i\in[m]: D_i\cdot f\in\zz\}\in\A_{\pm}\}$;
\item $\sigma_{\pm}=\theta_{\pm}(\sum_{i=1}^r\sfp_i^\pm\log(\sfy_i)+c_0(\lambda))
=\sum_{i=1}^{\ell_{\pm}}\theta_{\pm}(p_i^\pm)\log y_i-\sum_{j\in S_\pm}\lambda_j\cdot\log x_j +c_0(\lambda)$, 
where 
$$\theta_{\pm}: \ll^\vee\otimes\cc\to H_{\bbT}^2(X_\pm,\cc); \quad  \theta_{\pm}(D_i)=u_i-\lambda_i,$$
and $u_i$ is the cohomology class in $H_{\bbT}^2(X_\pm,\cc)$ Poincar\'e dual to the divisor classes 
$\{z_i=0\}$, $\{w_i\}=0$, or $\{u_j=0\}$.  Note that $u_j=0$ if $j\in S$. 
\item $\sfy^d=\sfy_1^{\sfp_1^\pm\cdot d}\cdots \sfy_r^{\sfp_r^\pm\cdot d}=\prod_{i=1}^{\ell_{\pm}}y_i^{p_i^\pm\cdot d}\prod_{j\in S_\pm}x_j^{D_j\cdot d}$.
\end{itemize}

The $I$-functions $I_\pm(\sfy,z)$ lie on the Givental's Lagrangian cone 
$\sL_{X_\pm}$ determined by genus zero Gromov-Witten invariants. 
From \cite{CIJ}, the $I$-functions $I_\pm(\sfy,z)$ are analytic in the last variable 
$\sfy_r$ and we do analytic continuation in terms of $\sfy_r$. 

In view of \cite[Proposition 6.4]{JT3}, we may identify\footnote{Since $Y_\pm\subset X_\pm$ is a complete intersection and the normal bundle $N_{Y_\pm/X_\pm}$ is trivial, this is just a simple example of orbifold quantum Riemann-Roch in genus $0$ \cite{Tseng}.} the cones $\sL_{Y_\pm}$ with the cones $\sL_{X_\pm}$. We simply define the $I$-functions $I^{Y_\pm}(\sfy,z)$ for $Y_\pm$ to be the $I$-functions $I_\pm(\sfy,z)$. Certainly they determine the cones $\sL_{Y_\pm}$.


\subsubsection{The analytic continuation}

Introduce the following modified Givental's spaces:
$$\widetilde{\sH}(X_\pm)=H_{\CR,\bbT}^*(X_\pm)\otimes_{R_{\bbT}}[\log z](\!(z^{-1/k})\!)$$ and 
$$\widetilde{\sH}(Y_\pm)=H_{\CR,\bbT}^*(Y_\pm)\otimes_{R_{\bbT}}[\log z](\!(z^{-1/k})\!)$$
where $k\in\nn$ is an integer such that $k\mu^\pm$ have integer eigenvalues for the grading operators $\mu^\pm$ in \cite[\S 2]{CIJ}. 
\begin{prop}[\cite{CIJ}, \S 6.2.4]
There is a symplectic transformation
$$\uu_{X}: \widetilde{\sH}(X_-)\to \widetilde{\sH}(X_+)$$
such that $\uu_X(I_-(\sfy,z))=I_+(\sfy,z)$.
\end{prop}
\begin{proof}
In \cite[Theorem 6.13]{CIJ}, the authors explicitly calculate the analytic continuation of the $H$-function, then the analytic continuation of the $I$-function in \cite[\S 6.2.4]{CIJ}.
\end{proof}

Our main result for $Y_\pm$ is as follows:
\begin{thm}\label{main:theorem}
There exists the following diagram:
\[
\resizebox{0.55\textwidth}{!}{
  \xymatrix{
    K_{\bbT}^0(X_-) \ar[rr]^{\FM=\Phi} \ar[rdd]^{\widetilde{\Gamma}_{-}} \ar[ddd]_{\iota_-^\star} && K_{\bbT}^0(X_+) \ar[rdd]^{\widetilde{\Gamma}_{+}} \ar'[dd]_-{\iota_+^\star}[ddd]\\
    \\
    & ~\widetilde{\sH}(X_-) \ar[rr]^<<<<<<<<<{\uu_{X}} \ar[ddd]_{\iota_-^\star} && ~\widetilde{\sH}(X_+) \ar[ddd]_{\iota_+^\star}\\
    K_{\bbT}^0(Y_-) \ar'[r]^-{\FM=\Psi}[rr] \ar[rdd]_{\widetilde{\Gamma}^{Y_-}_{-}} && K_{\bbT}^0(Y_+) \ar[rdd]_{\widetilde{\Gamma}^{Y_+}_{+}} \\
    \\
    &~ \widetilde{\sH}(Y_-) \ar@{->}[rr]^{\uu_Y} &&~ \widetilde{\sH}(Y_+) \\
  }
}
\]
where 
$$\widetilde{\Gamma}_{\pm}: K_{\bbT}^0(X_\pm)\to \widetilde{\sH}(X_\pm)$$
is defined by the $\Gamma_{\pm}(X_\pm)$-classes in Definition 3.1 of \cite{CIJ}, and 
$$\widetilde{\Gamma}^{Y_\pm}_{\pm}: K_{\bbT}^0(Y_\pm)\to \widetilde{\sH}(Y_\pm)$$
is defined by replacing the $\Gamma_{\pm}(X_\pm)$-classes by the $\Gamma_{\pm}(Y_\pm)$-classes.
The map
$$\uu_Y: \widetilde{\sH}(Y_-)\to \widetilde{\sH}(Y_+)$$
is a symplectic transformation on Givental's space for $Y_\pm$.  Moreover,
\begin{enumerate}
\item $\uu_Y(I^{Y_-}(\sfy,z))=I^{Y_+}(\sfy, z)$;
\item The upper square is a commutative diagram, which implies that the Fourier-Mukai transformation 
$\Phi$ matches the analytic continuation $\uu_X$ via the $\Gamma$-integral structure;
\item The bottom square is also commutative, which implies that the Fourier-Mukai transformation $\Psi$ matches the analytic continuation $\uu_Y$ via the $\Gamma$-integral structure. 
\end{enumerate}
\end{thm}
\begin{proof}
The proof is from \S \ref{I-functionXpm} and \cite[\S 7.3, \S7.4, \S7.5]{CIJ}. The difference here is that we can work on $\bbT$-equivariant K-theory and Chen-Ruan cohomology, not like \cite[\S 7]{CIJ} in the non-equivariant setting, but the calculation is the same. 
\end{proof}

\begin{rmk}
On the K\"ahler moduli space $\M$, the Fourier-Mukai transformation $\Psi$ is an equivalence
$$\Psi: D^b(Y_-)\to D^b(Y_+) \quad (K_{\bbT}^0(Y_-)\stackrel{\cong}{\rightarrow}K_{\bbT}^0(Y_+)).$$
From \cite[\S 6.5]{CIJ}, the Fourier-Mukai transformation 
$\Psi$ corresponds to a path $\gamma$ from the large radius point of $Y_-$ to the large radius point of $Y_+$ inside $\M$.  Let 
$$\Psi^\prime:=\FM^\prime=(F_-)_{\star}F_{+}^\star: K_{\bbT}^0(Y_+)\stackrel{\cong}{\rightarrow}K_{\bbT}^0(Y_-)$$
be the Fourier-Mukai transformation on the other side.  Then $\Psi^\prime\circ\Psi$
yields a loop in $\pi_1(\M)$, which gives rise to an automorphism of the $K$-theory group by
\begin{equation}\label{loop}
\rho: \pi_1(\M)\to \Aut(K_{\bbT}^0(Y_-)).
\end{equation}
On the other hand, the analytic continuation
$$\uu_Y:  \widetilde{\sH}(Y_-)\to \widetilde{\sH}(Y_+)$$
is given along a path $\gamma$ in Figure 3 of \cite{CIJ}, and 
hence also gives a loop of $\pi_1(\M)$ in (\ref{loop}).  Since the $I$-function $I_\pm(Y_\pm)$ determines the quantum connection, the above monodromy is the monodromy of the quantum connections. Theorem \ref{main:theorem} says that these two monodromies are the same, hence the monodromy conjecture in \cite{BMO}. 
\end{rmk}


\end{document}